\documentclass[a4paper,10pt]{article}

\usepackage{
	amsfonts, 
	amsmath, 
	amssymb, 
	amsthm, 
	fullpage, 
	xypic,
	hyperref,
	float,
}

\usepackage[USenglish]{babel}
\usepackage[first=0, last=1]{lcg}
\usepackage{authblk}

\newcommand{\diag}{\mathrm{diag}}
\newcommand\CC{{\mathbb C}}

\newcommand\NN{{\mathbb N}}

\newcommand\ad{\operatorname{ad}}

\newcommand{\dgr}{\mathrm{dgr}}

\newcommand{\Mat}{\mathrm{Mat}}

\def\pfq#1#2#3#4#5#6{%
	{}_{#1}\phi_{#2}\biggl[\genfrac..{0pt}{}{#3}{#4};#5,#6\biggr]%
}

\def\peq#1#2#3#4#5#6{%
	{}_{#1}\eta_{#2}\biggl[\genfrac..{0pt}{}{#3}{#4};#5,#6\biggr]%
}

\def\qbin#1#2{
	{#1 \brack #2}
}

\input xy
\xyoption{all}

\theoremstyle{definition}
\newtheorem{theorem}{Theorem}[section]
\newtheorem{proposition}[theorem]{Proposition}
\newtheorem{lemma}[theorem]{Lemma}
\newtheorem{definition}[theorem]{Definition}

\newtheorem{remark}[theorem]{Remark}
\newtheorem{corollary}[theorem]{Corollary}

\numberwithin{equation}{section}

\title{Matrix-Valued Little $q$-Jacobi Polynomials}
\author[1]{Noud Aldenhoven \thanks{n.aldenhoven@science.ru.nl}}
\author[1]{Erik Koelink \thanks{e.koelink@math.ru.nl}}
\author[2]{Ana M. de los R\'{i}os \thanks{amdelosrios@us.es}}
\affil[1]{Radboud Universiteit, IMAPP, FNWI\\
Heyendaalseweg 135, 6525 AJ Nijmegen,
the Netherlands.}
\affil[2]{Departamento de An\'{a}lisis Matem\'{a}tico,
       Universidad de Sevilla\\
Apdo (P. O. BOX) 1160, 41080 Sevilla, Spain}
\date{\emph{Dedicated to Dick Askey on the occasion of his 80th birthday \\
with admiration for Dick's achievements in special functions and mathematics education}}
\begin{document}

\maketitle

\begin{abstract}
Matrix-valued analogues of the little $q$-Jacobi polynomials are introduced and studied. 
For the $2\times 2$-matrix-valued little $q$-Jacobi polynomials explicit expressions for the orthogonality relations, Rodrigues formula, three-term recurrence relation and their relation to matrix-valued $q$-hypergeometric series and the scalar-valued little $q$-Jacobi polynomials are presented. 
The study is based on a matrix-valued $q$-difference operator, which is a $q$-analogue of Tirao's matrix-valued hypergeometric differential operator. 
\end{abstract}

\section{Introduction}

Matrix-valued orthogonal polynomials were originally introduced by M.G.Kre{\u\i}n in 1949, initially studying the corresponding moment problem, see references in \cite{Ber68, DPS08}, and to study differential operators and their deficiency indices, see also \cite{Kre49}. 
Since then an increasing amount of authors are contributing to build up a general theory of matrix-valued orthogonal polynomials (see for 
example \cite{Du96, Ge82, GPT03, KvPR12, SiVa96}, etc.).

In the study of matrix-valued orthogonal polynomials the general theory deals with obtaining appropriate analogues of classical results known for (scalar-valued) orthogonal polynomials, and many results and proofs have been generalized in this direction, see \cite{Du93, Du96} and the overview paper \cite{DPS08}. 
But also new features that do not hold in the scalar theory have been discovered, like the existence of different second order differential equations satisfied by a 
family of matrix orthogonal polynomials, see \cite{DD08, KvPR12}. 
The theory of matrix-valued orthogonal polynomials has also turned out to be a fruitful tool in the solution of higher order recurrence relations, 
see \cite{DuVa95, GIK13}.

For orthogonal polynomials the theory is complemented by many explicit families of orthogonal polynomials, notably the ones in the Askey scheme and its $q$-analogue, see \cite{KLS10, KS98}, which have turned out to be very useful in many different contexts, such as mathematical physics, representation theory, combinatorics, number theory, etc. 
The orthogonal polynomials in the ($q$-)Askey scheme are characterized by being eigenfunctions of a suitable second order differential or difference operator, so that all these families correspond to solutions of a bispectral problem. 
E.g., for the Jacobi polynomials this is the hypergeometric differential operator and for the little $q$-Jacobi polynomials this is the $q$-hypergeometric difference operator, see also \cite{GR04, Ism09}. 
This is closely related to Bochner's 1929 classification theorem of second order differential operators 
having polynomial eigenfunctions, see \cite{Ism09} for extensions and references.

For matrix-valued orthogonal polynomials there is no classification result of such type known, so that we have to study the properties of specific examples of families of matrix-valued orthogonal polynomials. 
Already many examples are known, either from scratch, \cite{DG04}, related to representation theory e.g. \cite{GPT02, KvPR12, KvPR13} or motivated from spectral theory \cite{GIK13}. 
In most of these papers, the matrix-valued orthogonal polynomials are eigenfunctions of a second order matrix-valued differential operator, 
so that these polynomials are usually considered as matrix-valued analogues of suitable polynomials from the Askey-scheme.
The matrix-valued differential operator is often of the type of the matrix-valued hypergeometric differential operator of Tirao \cite{Ti03} and this makes it possible to express matrix-valued orthogonal polynomials in terms of the matrix-valued hypergeometric functions, see e.g. \cite{KvPR13} for an example.

More recently, in \cite{ADdlR13} the step has been made to use matrix-valued difference operators and consider corresponding matrix-valued orthogonal polynomials as eigenfunctions. 
Again these matrix-valued orthogonal polynomials can be seen as analogues of orthogonal polynomials from the Askey-scheme.
In this paper, motivated by \cite{ADdlR13}, we study a specific case of matrix-valued orthogonal polynomials which are analogues of the 
little $q$-Jacobi polynomials, moving from analogues of classical discrete orthogonal polynomials to orthogonal polynomials on a $q$-lattice. 
As far as we are aware, these matrix-valued orthogonal polynomials are a first example of the matrix-valued analogue of a family
of polynomials in the $q$-Askey scheme.
An essential ingredient in the study of these matrix-valued little $q$-Jacobi polynomials is the second order $q$-difference operator (\ref{eqn:q-diff_eq}). 
In particular this gives the possibility to introduce and employ matrix-valued basic hypergeometric series in the same spirit as Tirao \cite{Ti03}, which differs from the approach of Conflitti and Schlosser \cite{CS10}. 

The content of the paper are as follows. 
In Section \ref{section:prelim} we recall the basics of the (scalar-valued) little $q$-Jacobi polynomials and the general theory of matrix-valued orthogonal polynomials.  
In Section \ref{section:qdiffop} we study the matrix-valued second order $q$-difference equations as well as 
under which conditions such an operator is symmetric for a suitable matrix-weight function. 
In Section \ref{section:qtirao} we study the relevant $q$-analogue of Tirao's \cite{Ti03} matrix-valued hypergeometric functions. 
In Section \ref{section:2x2case} the $2\times 2$-matrix-valued little $q$-Jacobi polynomials are studied in detail. 
In particular, we give explicit orthogonality relations, the moments, the matrix-valued three-term recurrence relation, expressions in terms of the matrix-valued basic hypergeometric function, the link to the scalar little $q$-Jacobi polynomials, and the Rodrigues formula for these family of polynomials.

It would be interesting to find a group theory interpretation of these matrix-valued little $q$-Jacobi polynomials along the lines of 
\cite{GPT02, KvPR12, KvPR13} in the quantum group setting.

\section{Preliminaries} \label{section:prelim}

\subsection{Basic hypergeometric functions}
We recall some of the definitions and facts about basic hypergeometric functions, see Gasper and Rahman \cite{GR04}.
We fix $0 < q < 1$.
For $a \in \mathbb{C}$ the $q$-Pochhammer symbol is defined recursively by $(a;q)_{0} = 1$ and
\begin{align*}
(a;q)_{n} &= (1 - aq^{n-1})(a;q)_{n-1}, 
\quad 
(a;q)_{-n} = \frac{1}{(aq^{-n}; q)_n},
\quad
n \in \NN = \{0, 1, 2, \ldots\}.
\end{align*}
The infinite $q$-Pochhammer symbol is defined as
\begin{align*}
(a;q)_{\infty} &= \prod_{k=0}^{\infty} (1-aq^{k}).
\end{align*}
For $a_1,\dots, a_{\ell} \in \mathbb{C}$ we use the abbreviation $(a_1, a_2, ..., a_{\ell}; q)_n = \prod_{i=1}^{\ell} (a_i;q)_{n}$.
The basic hypergeometric series ${}_{r+1}\phi_{r}$ with parameters $a_{1},\dots,a_{r+1}, b_{1},\dots,b_{r} \in \mathbb{C}$, base $q$ and variable $z$ is defined by the series
\begin{align*}
\pfq{r+1}{r}{a_1, a_2, \dots, a_{r+1}}{b_1, b_2, \dots, b_r}{q}{z}
	&= \sum_{k=0}^{\infty}
	\frac{
		(a_{1}, a_{2}, \dots, a_{r+1}; q)_{k}
	}{
		(q;q)_{k} (b_{1}, b_{2}, \dots b_{r}; q)_{k}
	}
	z^{k}, 
	\quad |z| < 1.
\end{align*}

The $q$-derivative $D_q$ of a function $f$ at $z \neq 0$ is defined by
\begin{align*} 
(D_{q}f)(z) &= \frac{f(z) - f(qz)}{(1-q)z},
\end{align*}
and $(D_{q}f)(0) = f'(0)$, provided that $f'(0)$ exists.
Two useful formulas are the $q$-Leibniz rule \cite[Exercise 1.12.iv]{GR04}
\begin{align}
\label{eqn:q-Leibniz}
D_q^n(fg)(z) &= \sum_{k=0}^n \qbin{n}{k}_q D_q^{n-k} f(q^{k}z) D_q^k g(z),
\end{align}
and the formula
\begin{align}
\label{eqn:Dqn:explicit}
(D_q^n f)(z) &= \frac{1}{(1 - q)^n q^{\binom{n}{2}} z^{n}}
	\sum_{j = 0}^n (-1)^{n-j} \qbin{n}{j}_q 
	q^{\binom{n-j}{2}} f(q^{j}z),
\end{align}
where the $q$-binomial coefficient $\qbin{n}{k}_q$ is given by
\begin{align*}
\qbin{n}{k}_q &= \frac{(q;q)_n}{(q;q)_k (q;q)_{n-k}}.
\end{align*}

The $q$-integral of a function $f$ is defined as
\begin{align*} 
\int_{0}^{1} f(z)d_{q}z
	&= (1-q) \sum_{k=0}^{\infty}f(q^{k})q^{k},
\end{align*}
whenever the series converges. 
The $q$-analogue of the fundamental theorem of calculus states 
\begin{align}
\label{eqn:q-fund_calc}
\int_{0}^{1} \bigl( D_q f\bigr)(z) d_{q}z &= \left. f(q^x) \right|_{x = \infty}^{0} = f(1) - f(0),
\end{align}
whenever all the limits converge.

\subsection{The little $q$-Jacobi polynomials}
Let $0 < a < q^{-1}$ and $b < q^{-1}$.
The little $q$-Jacobi polynomials are the polynomials defined by
\begin{align} \label{eqn:lqjp}
p_n(z;a,b;q) = \pfq{2}{1}{q^{-n}, abq^{n+1}}{aq}{q}{qz}.
\end{align}
The little $q$-Jacobi polynomials have been introduced by Andrews and Askey \cite{AA}, see also 
\cite[\S 7.3]{GR04} and \cite[\S 14.12]{KLS10}. 
These polynomials satisfy the following orthogonality relation
\begin{align}
\langle p_m(z,a,b;q), p_n(z,a,b;q) \rangle
	&= \sum_{k = 0}^{\infty}
		(aq)^k \frac{(bq; q)_{k}}{(q; q)_{k}}
		p_m(q^{k};a,b;q) p_n(q^{k};a,b;q) \label{eqn:lqjp_orth} \\
	\nonumber
	&= \frac{(abq^2; q)_{\infty}}{(aq; q)_{\infty}}
		\frac{(1 - abq)(aq)^n}{(1 - abq^{2n+1})}
		\frac{(q, bq; q)_n}{(aq, abq; q)_n} \delta_{m,n} 
	= h_{n}(a, b; q) \delta_{m,n}, 
\end{align}
where $\delta_{m,n}$ is the Kronecker delta function and $h_{n}(a, b; q) > 0$.
If we need to emphasize the dependence on $a$ and $b$ we write $\langle \cdot, \cdot \rangle_{(a, b)}$.
The moments of the little $q$-Jacobi polynomials are given by
\begin{align} \label{eqn:lqjp_moments}
m_n(a, b) &= \langle z^n, 1 \rangle_{(a, b)} 
	= \frac{(abq^{n+2}; q)_{\infty}}{(aq^{n+1}; q)_{\infty}}.
\end{align}
The sequence of the little $q$-Jacobi polynomials satisfies the three term recurrence relation
\begin{align} \label{eqn:recrellqJP}
-z p_n(z;a,b;q) &=
	A_n p_{n+1}(z;a,b;q) -
	(A_n + C_n) p_n(z;a,b;q) +
	C_n p_{n-1}(z;a,b;q)
\end{align}
with
\begin{align*}
A_n &= q^n
	\displaystyle \frac{(1 - aq^{n+1})(1 - abq^{n+1})}
	{(1 - abq^{2n+1})(1 - abq^{2n+2})}, \quad
C_n = aq^n
	\displaystyle \frac{(1 - q^n)(1 - bq^n)}
	{(1 - abq^{2n})(1 - abq^{2n+1})}.
\end{align*}
They are also eigenfunctions of the second order $q$-difference operator
\begin{align}\label{eq:diffeqlittleqJacobi} 
\lambda_n p_n(z)
	&= a(bq - z^{-1}) (E_{1}p_n)(z)
	- ((abq + 1) - (1 + a)z^{-1}) (E_{0} p_n)(z)
	+ (1 - z^{-1}) (E_{-1}p_n)(z),
\end{align}
where $\lambda_n = q^{-n}(1 - q^n)(1 - abq^{n+1})$, $p_n(z) = p_n(z;a,b;q)$ and $E_{\ell}$ are the $q$-shift operators defined by $(E_{\ell}p)(z) = p(q^{\ell}z)$.

\subsection{Matrix-valued orthogonal polynomials}
We review here some basic concepts of the theory of matrix-valued orthogonal polynomials, also see \cite{DPS08, GT07, Mir05}. 
A matrix-valued polynomial of size $N \in \NN$ is a polynomial whose coefficients are elements of $\Mat_N(\CC)$.
If no confusion is possible we will omit the size parameter $N$ and write $\mathbb{P}[z]$ for the space of matrix polynomials with coefficients in $\Mat_N(\CC)$ and $\mathbb{P}_n[z]$ for polynomials in $\mathbb{P}[z]$ of degree at most $n$.

The orthogonality will be with respect to a $N \times N$ weight matrix $W$, that is a matrix of Borel measures supported on a common set of the real line $\mathfrak{S}$, such that the following is satisfied:
\begin{enumerate}
\item for any Borel set $A \subseteq \mathfrak{S}$ the matrix $W(A) = \int_{A}dW(z)$ is positive semi-definite,
\item $W$ has finite moments of every order, i.e. $\int_{\mathfrak{S}} z^n dW(z)$ is finite for all $n \geq 0$,
\item if $P$ is a matrix-valued polynomial with non-singular leading coefficient then $\int_{\mathfrak{S}} P(z) dW(z) P^*(z)$ is also non-singular.
\end{enumerate}
A weight matrix $W$ defines a matrix-valued inner product on the space $\mathbb{P}[z]$ by
\begin{align*} 
\langle P, Q \rangle &= \int_{\mathfrak{S}} P(z) dW(z) Q^{*}(z) \in \Mat_N(\CC).
\end{align*}

Note that for every matrix-valued polynomial $P$ with non-singular leading coefficient, $\langle P, P \rangle$ is positive definite.
A sequence of matrix-valued polynomials $(P_{n})_{n \geq 0}$ is called orthogonal with respect to the weight matrix $W$ if
\begin{enumerate}
\item for every $n \geq 0$ we have $\dgr(P_n) = n$ and $P_{n}$ has non-singular leading coefficient,
\item for every $m, n \geq 0$ we have $\langle P_{m}, P_{n} \rangle = \Gamma_{m} \delta_{m, n}$, where $\Gamma_{m}$ is a positive definite matrix.
\end{enumerate}
Given a weight matrix $W$ there always exists a unique sequence of polynomials $(P_{n})_{n \geq 0}$ orthogonal with respect to $W$
up to left multiplication of each $P_n$ by a non-singular matrix, see \cite[Lemma 2.2 and Lemma 2.7]{DPS08} or \cite{GT07}.
We say that a matrix-valued orthogonal polynomials sequence $(P_{n})_{n \geq 0}$ is orthonormal if $\Gamma_n = I$ for all $n \geq 0$.
We call $(P_n)_{n \geq 0}$ monic if every $P_n$ is monic, i.e. the leading coefficient of $P_n$ is the identity matrix.

A weight matrix $W$ with support $\mathfrak{S}$ is said to be reducible to scalar weights if there exists a non-singular matrix $K$, independent of $z$, and a diagonal matrix $D(z) = \diag(w_{1}(z), w_{2}(z), \dots, w_{N}(z))$ such that for all $z \in \mathfrak{S}$
\begin{align*}
W(z) = K \, D(z) \, K^*.
\end{align*}
In this case the orthogonal polynomials with respect to $W(z)$ are of the form
\begin{align*}
P_{n}(z) &=
	\begin{pmatrix}
		p_{n,1}(z) & 0 & \cdots & 0 \\
		0 & p_{n, 2}(z) & \cdots & 0 \\
		\vdots & \vdots & \ddots & \vdots \\
		0 & 0 & \cdots & p_{n,N}(z)
	\end{pmatrix} K^{-1}
\end{align*}
where $(p_{n,i})_{n}$ are the orthogonal polynomials with respect to $D_{i,i}(z)=w_{i}(z)$ for $i = 1, \dots, N$.
Therefore weight matrices that reduce to scalar weights can be viewed as a set of independent scalar weights, so they are not interesting for the theory of matrix orthogonal polynomials.
In this paper $\mathfrak{S}$ is countable and assuming additionally that $W(a) = I$ for some $a \in \mathfrak{S}$, by \cite[Theorem 4.1.6]{HJ85} the weight matrix $W$ can be reduced to scalar weights if and only if $W(x) W(y) = W(y) W(x)$ for all $x, y \in \mathfrak{S}$, also see \cite[p. 43]{ADdlR13}.

In the rest of this paper we only consider weight matrices such that $dW(z) = \frac{1}{1 - q} W(z) d_{q}z$ and we assume that $W(q^n) > 0$ for all $n \in \NN$.
These weight matrices are called $q$-weight matrices or just $q$-weights.
The matrix-valued inner product defined by such a $q$-weight is of the form
\begin{align} \label{eqn:q-weight}
\langle P, Q \rangle_W 
	= \frac{1}{1-q} \int_{0}^{1} P(z) W(z) Q^*(z) d_{q}z
	= \sum_{n=0}^{\infty} q^{n} P(q^n) W(q^n) (Q(q^n))^*,
\end{align}
whenever the series converges termwise.

\section{$q$-Difference operators} \label{section:qdiffop}
In order to study matrix-valued analogues of the little $q$-Jacobi polynomials appearing in the $q$-Askey scheme we focus our attention on operators of the form
\begin{align} \label{eqn:q-diff_op}
D = E_{-1}F_{-1} + E_{0}F_{0} + E_{1}F_{1},
\end{align}
where $F_{\ell}(z)$ are matrix-valued polynomials in $z^{-1}$ satisfying certain degree conditions assuring the preservation of the polynomials, cf \eqref{eq:diffeqlittleqJacobi}. 
In particular we are interested in operators having families of matrix-valued polynomials as eigenfunctions,
\begin{align} \label{eqn:q-diff_eq}
(DP_{n})(z)
	&= P_n(q^{-1}z) F_{-1}(z) + P_n(z) F_{0}(z) + P(qz) F_{1}(z)
	= \Lambda_{n} P_{n}(z).
\end{align}
It is important to notice that the coefficients $F_{\ell}$ appear on the right whereas the eigenvalue matrix $\Lambda_{n}$ appears on the left, cf. \cite{Du97}.

\subsection{$q$-Difference operators preserving polynomials}

Suppose that there is a family of solutions of matrix-valued orthogonal polynomials to (\ref{eqn:q-diff_eq}), then $D$ 
preserves polynomials and does not raise the degree of a polynomial. 
Theorem \ref{thm:degree} characterizes the $q$-difference operators with polynomial coefficients in $z^{-1}$ preserving polynomials of degree $n$ for all $n$.
Theorem \ref{thm:degree} is an analogue of \cite[Lemma 2.2]{ADdlR13} and \cite[Lemma 3.2]{Du12}, where the proof is a slight adaptation of \cite[Lemma 2.2]{ADdlR13} and \cite[Lemma 3.2]{Du12}.

\begin{theorem} \label{thm:degree}
Let
\begin{align*}
D = \sum_{\ell = s}^{r}E_{\ell}F_{\ell}, \qquad F_{\ell} \in \mathbb{P}_{n}[z^{-1}]
\end{align*}
with $r,s$ integers such that $s \leq r$.
The following conditions are equivalent:
\begin{itemize}
\item[1.] $D\colon \mathbb{P}_{n}[z] \rightarrow \mathbb{P}_{n}[z]$ for all $n \geq 0$.
\item[2.] $F_{\ell}(z) \in \mathbb{P}_{r-s}[z^{-1}]$ for $\ell=s, \dots, r$ and $\sum_{\ell=s}^{r} q^{\ell k} F_{\ell}(z) \in \mathbb{P}_{k}[z^{-1}]$ for $k = 0, \dots, r-s$.
\end{itemize}
\end{theorem}

To prove Theorem \ref{thm:degree} we use Lemma \ref{lemma:gf}.

\begin{lemma} \label{lemma:gf}
Let $r, s$ and $n$ be integers such that $s \leq r$ and $0 \leq n$.
Let $G_{k}(z)$ be matrix-valued polynomials in $z^{-1}$ of degree at most $n$ for $k = 0,\dots,r-s$.
The system of linear equations
\begin{align*} 
\sum_{\ell =s}^{r} q^{\ell k}F_{\ell}(z) = G_{k}(z), \qquad 0 \leq k \leq r - s,
\end{align*}
determines the functions $F_{\ell}(z)$, $\ell = s, \ldots, r$, uniquely as polynomials in $z^{-1}$ of degree at most $n$.
\end{lemma}

The proof is straightforward, see \cite[Lemma 2.1]{Du12}, using the Vandermonde matrix.

\begin{proof}[Proof of Theorem \ref{thm:degree}]
First we prove $1 \Rightarrow 2$.
For $k=0, \dots, r-s$, let $G_{k}(z) = \sum_{\ell = s}^{r}q^{\ell k}F_{\ell}(z)$.
If $0 \leq n \leq r-s$, write $P(z)=z^{n}I$.
Then
\begin{align*}
(DP)(z) = \sum_{\ell = s}^{r}(q^{\ell}z)^{n} F_{\ell}(z)
	= z^{n}\sum_{\ell = s}^{r}q^{\ell n} F_{\ell}(z)
	= z^{n} G_{n}(z)\in \mathbb{P}_{n}[z].
\end{align*}
Since $DP$ is a polynomial with $\dgr(DP) \leq \dgr(P)$,  $G_{n}$ is a polynomial in $z^{-1}$ of degree at most $n$ for $0 \leq n \leq r-s$.
By Lemma \ref{lemma:gf} $F_{\ell}$ are actually polynomials in $z^{-1}$ of degree at most $r-s$.

To prove $2 \Rightarrow 1$, consider $G_{k}(z) = \sum_{\ell=s}^{r} q^{k\ell} F_{\ell}(z)$, for $k \geq 0$.
Now $G_{k}$ is a matrix-valued polynomial in $z^{-1}$.
If $0 \leq k \leq r-s$ then, by 2, $\dgr(G_k) \leq k$ and if $k \geq r-s+1$ then $\dgr(G_{k}) \leq r-s$.
Now put $P(z) = z^{n}I$ so
\begin{align*}
DP(z) = \sum_{\ell=s}^{r} (q^{\ell}z)^{n} F_{\ell}(z)
	= z^{n} \sum_{\ell=s}^{r}q^{\ell n} F_{\ell}(z).
\end{align*}
If $0 \leq n \leq r-s$ we have that $\sum_{\ell=s}^{r} q^{\ell n} F_{\ell}(z)$ is a polynomial in $z^{-1}$ of degree at most $n$.
Hence $DP$ is a polynomial of degree at most $n$.
On the other hand if $n \geq r-s$ then $\sum_{l=s}^{r} q^{\ell n} F_{\ell}(z)$ is a matrix-valued polynomial in $z^{-1}$ of degree at most $r-s$.
Hence $DP$ is also a polynomial in $z$ of degree at most $n$.
\end{proof}

\subsection{The symmetry equations}

Symmetry is a key concept when looking for weight matrices having matrix-valued orthogonal polynomials as eigenfunctions of a suitable $q$-difference operator.
In Definition \ref{def:symmetry} we use the notation \eqref{eqn:q-weight}.

\begin{definition}\label{def:symmetry}
An operator $D \colon \mathbb{P}[z] \to \mathbb{P}[z]$ is symmetric with respect to a weight matrix $W$ if $\langle DP, Q \rangle_W = \langle P, DQ \rangle_W$ for all $P, Q \in \mathbb{P}[z]$.
\end{definition}

Theorem \ref{thm:sym_eigen} is a well-known result relating symmetric operators and matrix-valued orthogonal polynomials, see \cite[Lemma 2.1]{Du97} and 
\cite[Proposition 2.10 and Corollary 4.5]{GT07}.

\begin{theorem} \label{thm:sym_eigen}
Let $D$ be a $q$-difference operator preserving $\mathbb{P}[z]$, so that $\dgr(DP) \leq \dgr(P)$ for any $P \in \mathbb{P}[z]$.
If $D$ is symmetric with respect to a weight matrix $W$ then there exists a sequence of matrix-valued orthonormal polynomials $(P_n)_{n \geq 0}$ and a sequence of Hermitian matrices $(\Lambda_n)_{n \geq 0}$ such that 
\begin{equation} \label{eqn:sym_eigen}
DP_{n}= \Lambda_{n} P_{n}, \quad \forall n \geq 0.
\end{equation}
Conversely if $W(z)$ is a weight matrix and $(P_{n})_{n \geq 0}$ a sequence of matrix-valued orthonormal polynomials such that there exists a sequence of 
Hermitian matrices $(\Lambda_n)_{n \geq 0}$ satisfying (\ref{eqn:sym_eigen}), then $D$ is symmetric with respect to $W$.
\end{theorem}

It should be observed that Theorem \ref{thm:sym_cond} is an analogue of a similar statement for differential operators in \cite{Du97, GT07}.
Also note that the $q$-difference operator $D$ has has polynomial coefficients in $z^{-1}$ (instead of $z$), so that the essential condition is the preservation of the space of polynomials instead of the degree condition, which is the essential condition in \cite{Du97, GT07}.

Theorem \ref{thm:sym_cond} is an analogue of \cite[Theorem 3.1]{DG04} and \cite[Section 4]{GPT03} for symmetric second order differential operators.

\begin{theorem}\label{thm:sym_cond}
Let $D$ be a $q$-difference operator preserving $\mathbb{P}[z]$ of the form (\ref{eqn:q-diff_op}) with $F_{\ell}$ matrix-valued polynomials in $z^{-1}$.
Let $W$ be a $q$-weight matrix as in (\ref{eqn:q-weight}).
Suppose that the coefficients $F_{\ell}$ and the weight matrix $W$ satisfy the following equations
\begin{align}
\label{eqn:sym_eqn0}
F_0(q^{x}) W(q^{x}) &= W(q^{x}) F_0(q^{x})^*, \quad x \in \NN, \\
\label{eqn:sym_eqn1}
F_1(q^{x-1}) W(q^{x-1}) &= q W(q^{x}) F_{-1}(q^x)^*,
	\quad x \in \NN \backslash  \{0\},
\end{align}
and the boundary conditions
\begin{align}
\label{eqn:bnd_eqn}
W(1) F_{-1}(1)^* &= 0, \\
\nonumber
q^{2x}F_{1}(q^{x}) W(q^{x}) &\to 0,\quad  \text{ as } x \to \infty, \\
\nonumber
q^x \bigl( F_1(q^x) W(q^x) - W(q^x) F_1(q^x)^* \bigr) &\to 0, 
	\quad \text{ as } x \to \infty.
\end{align}
Then the $q$-difference operator $D$ is symmetric with respect to $W$.
\end{theorem}

\begin{proof}
We assume that the operator $D$ and the weight matrix $W$ satisfy the symmetry equations (\ref{eqn:sym_eqn0}), (\ref{eqn:sym_eqn1}) and the boundary conditions (\ref{eqn:bnd_eqn}).
For an integer $M > 0$, we consider the truncated inner product
\begin{align*}
\left\langle P, Q \right\rangle_W^M &= \sum_{x = 0}^M q^x P(q^x) W(q^x) Q^{*}(q^x).
\end{align*}
It is clear that $\langle P, Q \rangle_W^M \to \langle P, Q \rangle_W$ as $M \to \infty$ for $P, Q \in \mathbb{P}[z]$.
Then 
\begin{align*}
\left\langle DP, Q \right\rangle_W^{M} 
	&= \sum_{x=0}^{M} q^{x} (DP)(q^{x}) W(q^x) Q^*(q^{x}) \\
	&= \sum_{x=0}^{M} q^{x} \left(
		P(q^{x+1}) F_{1}(q^x) + P(q^{x}) F_{0}(q^x) 
			+ P(q^{x-1}) F_{-1}(q^x)
	\right) W(q^x) Q^*(q^{x})
\end{align*}
By a straightforward and careful computation using (\ref{eqn:sym_eqn0}), (\ref{eqn:sym_eqn1}) and the first boundary condition of (\ref{eqn:bnd_eqn}) we have
\begin{align*}
\left\langle DP,Q \right\rangle_W^M - \left\langle P,DQ \right\rangle_W^M 
	&= q^{M} P(q^{M+1}) F_{1}(q^{M}) W(q^{M}) Q^*(q^{M})
	- q^{M} P(q^{M}) W(q^{M}) F_{1}^*(q^{M}) Q^*(q^{M+1}).
\end{align*}
Write $P(z) = P_0 + zP_1(z)$ and $Q(z) = Q_0 + zQ_1(z)$ so that 
\begin{align*}
\left\langle DP,Q \right\rangle_W^{M}
	- \left\langle P,DQ \right\rangle_W^{M} 
	&= q^M P_0 \left( F_1(q^M) W(q^M) - W(q^M) F_1(q^M)^* \right) Q_0^*
	+ \text{remainder}
\end{align*}
where the remainder consists of terms of the form $q^{2M}R(q^M) F_1(q^M)W(q^M)S(q^M)$ or its adjoint for suitable matrix-valued polynomials $R$ and $S$. 
Taking $M \rightarrow \infty$ and using the last two boundary conditions of (\ref{eqn:bnd_eqn}) we get the result.
\end{proof}

\section{A matrix-valued $q$-hypergeometric equation} \label{section:qtirao}
Motivated by Tirao \cite{Ti03} we define a matrix-valued analogue of the basic hypergeometric series.
This definition is different from that given by Conflitti and Schlosser \cite{CS10}, where some additional factorization is assumed. 

Consider the following $q$-difference equation on row-vector-valued functions $F : \CC \to \CC^N$.
\begin{align} \label{eqn:tirao_diff}
F(q^{-1}z)(R_1 + z R_2) + F(z)(S_{1} + z S_{2}) + F(qz)(T_{1} + zT_{2}) = 0.
\end{align}
where $R_1, R_2, S_{1}, S_{2}, T_{1}, T_{2} \in \Mat_N(\mathbb{C})$. The case $N = 1$ is the scalar hypergeometric $q$-difference equation, see \cite[Exercise 1.13]{GR04}.

Let $F$ be a solution of (\ref{eqn:tirao_diff}) of the form $F(z) = z^{\mu} G(z)$ where 
\begin{align*}
G(z) &= \sum_{k=0}^{\infty} G^{k} z^k, \quad G^0 \neq 0, \ G^k\in \mathbb{C}^N
\end{align*}
The Frobenius method gives the recursions
\begin{align*}
0 &= G^{0} \left(q^{-\mu} R_1 + S_{1} + q^{\mu} T_{1} \right),\\
0 &= G^{k} \left(q^{-k - \mu} R_1 + S_{1} + q^{k + \mu} T_{1} \right)
	+ G^{k-1} \left(
		q^{-k - \mu + 1} R_2 + S_{2} + q^{k + \mu - 1} T_{2}
	\right), \quad k \geq 1.
\end{align*}
The first equation implies $\det(q^{-\mu}R_1 + S_{1} + q^{\mu}T_{1}) = 0$ and $\left(G^{0}\right)^* \in \ker(q^{-\bar\mu}R^*_1 + S^*_{1} + q^{\bar\mu}T^*_{1})$.
The solution of the indicial equation $\det(q^{-\mu}R_1 + S_{1} + q^{\mu}T_{1}) = 0$ is the set of exponents $E$.
For each $\mu \in E$ we write $d_{\mu} = \dim(\ker(q^{-\bar\mu} R^*_1 + S^*_1 + q^{\bar\mu} T^*_1))$ for the multiplicity of the exponent $\mu$. 
In order to have analytic solutions of (\ref{eqn:tirao_diff}) we require that $0 \in E$. 
Moreover we assume that the multiplicity for $0$ is maximal, $d_{0} = N$, which implies $S_1 = -R_1 - T_1$.
Under this assumption $E = \{ \mu : \det(q^{-\mu} R_1 - T_1) = 0 \} \cup \{ 0 \}$.
Since we are only interested in polynomial solutions, we only consider expansions around $z = 0$, but we can also study solutions at $\infty$ in a similar fashion.

We specialize to the case $R_1 = -R_2 = I$; 
\begin{align} \label{eqn:tirao_ours}
F(q^{-1}z)(1 - z) + F(z)(-I - T_1 + z S_{2}) + F(qz)(T_{1} + z T_{2}) = 0.
\end{align}
For any $G^0 \in \CC^N$, $G(z) = \sum_{k = 0}^{\infty} G^k z^k$ is a solution of (\ref{eqn:tirao_ours}) if and only if
\begin{align*}
0 &= G^{k} \left((q^{-k} - 1)I + (q^{k} - 1)T_{1} \right)
	+ G^{k-1} \left(-q^{-k + 1}I + S_{2} + q^{k-1} T_{2} \right),
	\quad k \geq 1.
\end{align*}
Assuming $\sigma(T_1) \cap q^{-\NN}=\emptyset$, the coefficients are
\begin{align*}
G^{k} &= \frac{q^k}{(q;q)_k}
	G^{0}
	\prod_{i = 1}^{\substack{k \\ \longrightarrow}}
	\left(
		I - q^{i-1} S_{2} - q^{2i - 2} T_{2}
	\right) \left(
		I - q^{i} T_{1}
	\right)^{-1}, \quad k \geq 1,
\end{align*}
where $\displaystyle \prod_{i = 1}^{\substack{k \\ \longrightarrow}} A_i = A_1 A_2 \dots A_k$ is an ordered product.
We summarize this discussion with Definition \ref{def:prod} and Theorem \ref{thm:MatqHyp}.

\begin{definition} \label{def:prod}
Let $A, B, C \in \Mat_N(\CC)$ where $\sigma(C) \cap q^{-\mathbb{N} \backslash \{0\}} = \emptyset$.
Define
\begin{align*}
\left(A, B; C; q \right)_{0} &= I,\\
\left(A, B; C; q \right)_{k}
	&= \left(A, B; C; q \right)_{k-1}
	\left(
		I - q^{k-1} A - q^{2k-2} B
	\right) \left(
		I - q^{k} C
	\right)^{-1}, & k &\geq 1.
\end{align*}
Define the function ${}_2 \eta_1$ by
\begin{align} \label{eqn:MVbhs}
\peq{2}{1}{A, B}{C}{q}{z} = \sum_{n=0}^{\infty}
	(A, B; C; q)_n \frac{z^{n}}{(q;q)_n}.
\end{align}
\end{definition}

Now (\ref{eqn:MVbhs}) converges for $|z| < 1$ in the norm of $\Mat_N(\CC)$.

\begin{theorem} \label{thm:MatqHyp}
Let $A, B, C \in \Mat_N(\CC)$ such that $\sigma(C) \cap q^{-\mathbb{N} \backslash \{0\}} = \emptyset$.
\begin{align} \label{eqn:MatqHyp}
F(z) = F^0\, \peq{2}{1}{A, B}{C}{q}{qz}, \qquad F^0 \in \CC^N \text{ as row-vector},
\end{align}
is a solution of the matrix-valued $q$-difference equation
\begin{align} \label{eqn:tirao2}
F(q^{-1}z)(1 - z) + F(z) \left(-I - C + zA \right)
	+ F(qz) \left(C + zB \right) &= 0,
\end{align}
with condition $F(0) = F^{0}$.
Conversely, any analytic solution $F$ around $z=0$ of (\ref{eqn:tirao2}) with initial condition $F(0) \neq 0$ is of the form (\ref{eqn:MatqHyp}).
\end{theorem}

\section{The $2\times 2$ matrix-valued little $q$-Jacobi polynomials} \label{section:2x2case}

Based on \cite[Theorem 4.2]{ADdlR13} we present a method to construct $q$-difference operators $D$ and weight matrices $W$ satisfying the symmetry equations (\ref{eqn:sym_eqn0}). 
By applying this method we construct a matrix analogue of the scalar little $q$-Jacobi polynomials, and we give explicit expressions 
of the matrix-valued little $q$-Jacobi polynomials in terms of the scalar ones. We also show how these matrix polynomials can be written as a 
matrix-valued $q$-hypergeometric function, motivated by the work of Tirao \cite{Ti03}.

\subsection{The construction}
Lemma \ref{lem:construction} is an adapted version of \cite[Theorem 4.2]{ADdlR13}. 
We omit the proof here because it is completely analogous to that in \cite{ADdlR13}.

\begin{lemma} \label{lem:construction}
Let $s$ be a scalar function satisfying $s(q^{x}) \neq 0$ for $x \in \NN \backslash \{ 0 \}$.
Assume that $F_{1}$ and $F_{-1}$ are matrix-valued polynomials satisfying
\begin{align} \label{eqn:scalar_matrix}
F_{1}(q^{x - 1}) F_{-1}(q^{x}) &= q|s(q^x)|^{2} I,
	\quad \forall x \in \NN \backslash \{ 0 \}.
\end{align}
Let $T$ be a solution of the $q$-difference equation 
\begin{align} \label{eqn:Tn}
T(q^{x-1}) &= s(q^x)^{-1} F_{-1}(q^{x}) T(q^{x}),
	\quad x \in \NN \backslash \{ 0 \}, \quad T(1) = I.
\end{align}
Then the $q$-weight defined by $W(q^x) = T(q^x) T(q^x)^{*}$ satisfies the symmetry equation
\begin{align*}
F_{1}(q^{x - 1}) W(q^{x - 1}) &= q W(q^{x}) F_{-1}(q^{x})^{*},
	\quad x \in \NN \backslash \{ 0 \}.
\end{align*}
\end{lemma}

We are now ready to introduce $2 \times 2$ matrix-valued orthogonal polynomials related to a  specific $q$-difference operator.

\begin{theorem} \label{thm:eigenfunctions-values}
Assume $a$ and $b$ satisfy $0 < a < q^{-1}$ and $b < q^{-1}$. 
For $v \in \mathbb{C}$ define matrices
\begin{align*}
K &= \begin{pmatrix}
	0 & v(1 - q)(q^{-1} - a) \\
	0 & 0
\end{pmatrix}, &
M &= \begin{pmatrix}
	1 & v \\
	0 & 0
\end{pmatrix}, &
A &= e^{\log(q) M}
=
\begin{pmatrix}
	q & -v(1 - q) \\
	0 & 1
\end{pmatrix}.
\end{align*}
The $q$-difference operator given by
\begin{align} 
D &= E_{-1} F_{-1} (z) + E_{0} F_{0} (z) + E_{1} F_{1} (z) , \nonumber \\
F_{-1}(z) = (z^{-1} - 1) A^{-1}, \quad
&F_{0}(z) = K - z^{-1}(A^{-1} + aA), \quad
F_{1}(z) = (az^{-1} - abq)A, \label{eqn:main-qdiff}
\end{align}
is symmetric with respect to the matrix-valued inner product (\ref{eqn:q-weight}) where the weight matrix is given by
\begin{align} \label{eqn:mvlqjp-measure}
W(q^{x}) &=
	a^{x} \frac{(bq;q)_x}{(q;q)_{x}} A^{x} (A^{*})^{x},
\end{align}
Moreover, the monic orthogonal polynomials $(P_n)_{n \geq 0}$ with respect to $W$ satisfy $D P_n = \Lambda_n P_n$, with eigenvalues
\begin{align} \label{eqn:Gamma_n}
\Lambda_{n} &= \begin{pmatrix}
	-q^{-n-1} - abq^{n+2}	& v(1 - q)(abq^{n + 1} - q^{-1-n} + q^{-1} - a) \\
	0			& -q^{-n} - abq^{n + 1}
\end{pmatrix}.
\end{align}
\end{theorem}

\begin{remark}
These polynomials are a matrix-valued analogues of the little $q$-Jacobi polynomials and for $v\neq 0$ they cannot be reduced to scalars. 
This follows from $W(q^0) = I$ and for $v \neq 0$ and $x, y \in \mathbb{N} \backslash \{0\}$ with $x \neq y$ we have $W(q^x) W(q^y) \neq W(q^y) W(q^x)$, as can be checked by substituting (\ref{eqn:AxAx*}) in (\ref{eqn:mvlqjp-measure}).
\end{remark}

\begin{proof}
To prove the theorem we first prove that $D$ preserves polynomials and then apply Theorem \ref{thm:sym_cond} to see that the operator is symmetric with respect to $W$. 
We proceed in three steps.
 
\textit{Step 1. $D$ preserves polynomials and degree.}\\
As polynomials in $z^{-1}$, $\dgr(F_{i}) = 1 < 2$, so that condition $2$ for $k = 1, 2$ in Theorem \ref{thm:degree} is satisfied.
Because $\dgr(F_{0} + F_{1} + F_{-1}) = 0$ condition $2$ is also satisfied for $k = 0$ from which we conclude from Theorem \ref{thm:degree} that 
$D \colon \mathbb{P}_{n}[z] \rightarrow \mathbb{P}_{n}[z]$. 

\textit{Step 2. Symmetry equations. $F_1 ( q^{x-1} ) W( q^{x-1} ) = q W ( q^x ) F_{-1} ( q^{x} )$ and $F_0( q^x ) W (q^x) =W (q^x) F_0 ( q^x )^*$}\\
Consider the function $s(q^{x}) = q^{-x} \sqrt{a(1 - q^{x})(1 - bq^{x})}$. 
The function $s$ satisfies $s(q^x)\neq 0$ because $0 < a < q^{-1}$ and $b < q^{-1}$, and by a direct computation we see that with this choice of $s$, (\ref{eqn:scalar_matrix}) is satisfied.
The solution of (\ref{eqn:Tn}) in this case is given by
\begin{align*}
T(q^{x}) = \sqrt{\frac{a^{x} (bq;q)_{x}}{(q;q)_{x}}} A^{x},
\end{align*}
By Lemma \ref{lem:construction} we can conclude that the symmetry equation, $F_1 ( q^{x-1} ) W( q^{x-1} ) = q W ( q^x ) F_{-1} ( q^{x} )$ holds for $W(q^x) = T(q^x)T(q^x)^*$ and $x = 1, 2, \ldots$ .

Note that $T(q^x)$ is invertible for all $x\in\mathbb{N}$. The symmetry equation (\ref{eqn:sym_eqn0}) is equivalent to showing that the matrix $T(q^{x})^{-1} F_0(q^x) T(q^x)$ is Hermitian.
Note that
\begin{align} \label{eqn:AKA}
T(q^{x})^{-1} F_{0}(q^{x}) T(q^{x})
	&= A^{-x} \left(K - q^{-x} \left( A^{-1} + aA \right) \right) A^{x}
	= A^{-x} K A^{x} - q^{-x} \left( A^{-1} + aA \right).
\end{align}
Taking into account that $A = e^{\log(q)M}$ and $e^{-Nx} R e^{Nx} = \displaystyle \sum_{k=0}^{\infty} \frac{(-1)^{k} x^{k}}{k!} \ad_{N}^{k} R$, we 
see that \eqref{eqn:sym_eqn0} holds if and only if
\begin{align*}
A^{-x} K A^{x} - q^{-x} \left( A^{-1} + aA \right)
	&= \sum_{k=0}^{\infty}
		\frac{(-1)^{k} \log(q)^{k} x^{k}}{k!}
		\left(\ad_{M}^{k}K- A^{-1} - aA\right)
\end{align*}
is Hermitian for all $x\in \mathbb{N}$, i.e. if and only if all coefficients $\ad_{M}^{k}K - \left( A^{-1} + aA \right)$ are Hermitian.

For $k = 0$
\begin{align*}
V = K - aA - A^{-1} = \begin{pmatrix}
	-aq-q^{-1} & 0 \\
	0          & -a - 1
\end{pmatrix}
\end{align*}
is Hermitian.
Direct computation shows that $V$ satisfies
\begin{align*}
\ad_{M}V
	= M V - V M
	= \begin{pmatrix}
		-aq-q^{-1}	& 0 \\
		0		& -a - 1
	\end{pmatrix}
	= V
\end{align*}
i.e., $V$ is a fixed point of $\ad_M$.

On the other hand since $A = e^{-\log(q)M}$, we have $\ad_{M}^{k} K = \ad_{M}^{k} (K-aA-A)=\ad_{M}^{k} V$, so we get that 
$T(q^{x})^{-1} F_{0}(q^{x}) T(q^{x}) = V q^{-x}$, which is a diagonal real matrix, hence Hermitian.

\textit{Step 3. Boundary conditions. }\\
Since $F_{-1}(z) = (z^{-1}-1)A^{-1}$, the first boundary condition $F_{-1}(1)W(1)=0$ holds. 

To check the last boundary condition $q^x(F_1(q^x)W(q^x) - W(q^x)F_1(q^x)^*) \to 0$ as $x \to \infty$, we calculate
\begin{align} \label{eqn:Ax}
A^x &= \begin{pmatrix}
	q^{x} & -v(1 - q^{x}) \\
	0 & 1
\end{pmatrix}, \quad x\in\mathbb{Z}.
\end{align}
Then we have
\begin{align} \label{eqn:last_sym_check}
q^x(F_1(q^x)W(q^x) - W(q^x)F_1(q^x)^*)
	&= q^x a^x \frac{(bq;q)_x}{(q;q)_x} (aq^{-x} - abq) \left(
		A^{x+1} (A^*)^x - A^x (A^*)^{x+1}
	\right) \\
	&= q^{2x-1} a^x \frac{(bq;q)_x}{(q;q)_x} (aq^{-x} - abq)
		A 
		\begin{pmatrix}
			0 & -v(1 - q) \\
			\bar{v}(1 - q) & 0
		\end{pmatrix}
		A^*. \nonumber
\end{align}
Because $a < q^{-1}$, (\ref{eqn:last_sym_check}) tends to $0$ if $x \rightarrow \infty$.
It is easy to see that the second boundary condition $q^{2x}F_{1}(q^x) W(q^x) \to 0$ as $x \to \infty$ also holds.

We have proved that $D \colon \mathbb{P}[z] \rightarrow \mathbb{P}[z]$ is symmetric with respect to $W$ and that $D$ is an operator that preserves polynomials and does not raise the degree. 
We are under the hypothesis of Theorem \ref{thm:sym_eigen}, and we can conclude that the orthogonal polynomials with respect to $W$ are common eigenfunctions of $D$. 

Finally by equating the coefficients in the equation $DP_{n}= \Lambda_{n}P_{n}$, for the monic sequence of orthogonal polynomials with respect to $W$, 
we obtain the expression (\ref{eqn:Gamma_n}) for the eigenvalues $\Lambda_{n}=-q^{-n}A^{-1}+K-abq^{n+1}A$.
\end{proof}

\begin{proposition}
The moments associated to (\ref{eqn:q-weight}) with $W$ as in (\ref{eqn:mvlqjp-measure}) are given by
\begin{align*}
M_n &= \langle z^n I, I \rangle_{W}
=
\begin{pmatrix}
	m_n(aq^2, b) & -v(m_n(a, b) - m_n(aq, b)) \\
	-\overline{v}(m_n(a, b) - m_n(aq, b)) & m_n(a, b)
\end{pmatrix},
\end{align*}
where $m_n(a, b)$ are given in (\ref{eqn:lqjp_moments}).
\end{proposition}

\begin{proof}
Using (\ref{eqn:Ax}) we can write
\begin{align} \label{eqn:AxAx*}
A^x (A^*)^x &= \begin{pmatrix}
	q^{2x} & -v(1 - q^x) \\
	-\overline{v}(1 - q^x) & 1
\end{pmatrix}.
\end{align}
Substituting (\ref{eqn:AxAx*}) in  $\langle x^n I, I \rangle_W = \sum_{x = 0}^{\infty} q^x a^x \frac{(bq;q)_x}{(q;q)_x}q^{nx} A^x (A^*)^x$, we get the result 
using \eqref{eqn:lqjp_moments}.
\end{proof}

\subsection{Explicit expression of $P_n$}
In this section we give an explicit expression of $P_n$ in terms of scalar little $q$-Jacobi polynomials by decoupling the matrix-valued 
$q$-difference operator (\ref{eqn:main-qdiff}).

\begin{theorem} \label{thm:exp_exp}
The monic matrix-valued orthogonal polynomials with respect to the matrix-valued inner product \eqref{eqn:q-weight} and weight matrix \eqref{eqn:mvlqjp-measure} are of the form
\begin{align*}
P_n(z) &= 
N_n^{-1}
\begin{pmatrix}
	\kappa_{11}^n p_n(z; aq^2, b; q) & 
		\kappa_{12}^n p_{n+1}(z; a, b; q) 
		+ \kappa_{11}^n (1 - z) v p_n(z; aq^2, b; q) \\
	\kappa_{12}^n p_{n-1}(z; aq^2, b; q) &
		\kappa_{22}^n p_n(z; a, b; q)
		+ \kappa_{21}^n (1 - z) v p_{n-1}(z; aq^2, b; q)
\end{pmatrix},
\end{align*}
where 
\begin{align*}
N_n &= \begin{pmatrix}
1 & \alpha \\
0 & 1
\end{pmatrix}, \quad \alpha =  \cfrac{1 - q^n + aq^{n+1} - abq^{2n + 2}}{1 - abq^{2n+2}} v = v\left(1+q^n\frac{aq-1}{1-abq^{2n+2}}\right),
\end{align*}
$p_n(z; a, b; q)$ are the little $q$-Jacobi polynomials (\ref{eqn:lqjp}) and with coefficients
\begin{align}\label{initial_values}
\kappa_{11}^n &= (-1)^n q^{\binom{n}{2}} \frac{(aq^3; q)_n}{(abq^{n+3}; q)_n}, \qquad
\kappa_{12}^n = (-1)^{n+1} v q^{\binom{n+1}{2}} \frac{
	(aq; q)_{n+1}
}{
	(abq^{n+2}; q)_{n+1}
}, \\[0.4cm]
\kappa_{21}^n &= (-1)^n \xi_n a \overline{v} q^{\binom{n}{2}-n+2} \frac{
	(1 - q^n)(1 - bq^n)
}{
	(1 - aq)(1 - aq^2)
} \frac{
  (aq; q)_n
}{
  (abq^{n+1}; q)_n
}, \quad
\kappa_{22}^n = (-1)^n \xi_n q^{\binom{n}{2}} \frac{
  (aq; q)_n
}{
  (abq^{n+1}; q)_n
},\nonumber
\end{align}
where
\begin{align*}
\xi_n &= \left(1 + aq|v|^2 \frac{
	(1 - q^n)(1 - bq^n)
}{
	(1 - abq^{n+1})(1 - aq^{n+1})
} \right)^{-1}.
\end{align*}
\end{theorem}

\begin{proof}
Let us define $\tilde{P}_{n} = N_{n} P_{n}$ and notice that $(\tilde{\Lambda}_{n})_{n \geq 0} = (N_{n} \Lambda_{n} N_{n}^{-1})_{n \geq 0}$ are diagonal. Then $D\tilde{P}_{n} = \tilde{\Lambda}_n \tilde{P}_{n} = \diag(-q^{-n-1} - abq^{n+2}, -q^{-n} - abq^{n+1})\tilde{P}_{n}$.
Now write using \eqref{eqn:Ax}
\begin{align}
\nonumber
\tilde{P}_n(q^x) &= N_n P_n(q^x) = \begin{pmatrix}
	\tilde{p}^n_{11}(q^x) & \tilde{p}^n_{12}(q^x) \\
	\tilde{p}^n_{21}(q^x) & \tilde{p}^n_{22}(q^x),
\end{pmatrix}, \\
\label{eqn:coefs_Qn}
Q_n(q^x)
&= \tilde{P}_n(q^x) A^x
= \begin{pmatrix}
	q^x \tilde{p}^n_{11}(q^x) & 
		\tilde{p}^n_{12} - (1 - q^x) v \tilde{p}^n_{11}(q^x) \\
	q^x \tilde{p}^n_{21}(q^x) & 
		\tilde{p}^n_{22} - (1 - q^x) v \tilde{p}^n_{21}(q^x)
\end{pmatrix}
= \begin{pmatrix}
	r^n_{11}(q^x) & r^n_{12}(q^x) \\
	r^n_{21}(q^x) & r^n_{22}(q^x)
\end{pmatrix}, \\
\nonumber
\Longrightarrow &\quad
\begin{pmatrix}
	\tilde{p}^n_{11}(q^x) & \tilde{p}^n_{12}(q^x) \\
	\tilde{p}^n_{21}(q^x) & \tilde{p}^n_{22}(q^x)
\end{pmatrix} = \begin{pmatrix}
	q^{-x} r^n_{11}(q^x) & r^n_{12}(q^x) + v(q^{-x}-1) r^n_{11}(q^x) \\
	q^{-x} r^n_{21}(q^x) & r^n_{22}(q^x) + v(q^{-x}-1) r^n_{21}(q^x)
\end{pmatrix},
\end{align}
Taking into account (\ref{eqn:AKA}) and the proof of step 2 in the proof of Theorem \ref{thm:eigenfunctions-values} we obtain
\begin{align} \label{eqn:DQn}
(D\tilde{P}_n)(q^x) A^x
  &= \tilde{P}_n(q^{x-1}) A^x (q^{-x} - 1) A^{-1}
	+ \tilde{P}_n(q^x) A^x \left( A^{-x} K A^x - q^{-x}(A^{-1} + aA) \right) \nonumber \\
	&\qquad + \tilde{P}_n(q^{x+1}) A^x (aq^{-x} - abq) A \nonumber \\
	&= 
	(q^{-x} - 1) Q_n(q^{x-1})
	+ Q_n(q^x) q^{-x} \begin{pmatrix} 
		-(q^{-1} + aq) & 0 \nonumber \\
		0 & -(1 + a)
	\end{pmatrix}
	+ (aq^{-x} - abq) Q_n(q^{x+1})  \\
	&= 
	\diag(-q^{-n-1} - abq^{n+2}, -q^{-n} - abq^{n+1}) Q_{n}(q^x).
\end{align}
Since the eigenvalues as well as all the matrix coefficients involved are diagonal, (\ref{eqn:DQn}) gives four uncoupled scalar-valued $q$-difference equations 
\begin{align*}
r^n_{11}(q^{x-1})&(q^{-x} - 1) - r^n_{11}(q^{x})q^{-x}(q^{-1} + aq)
	+ r^n_{11}(q^{x+1})(aq^{-x} - abq) = -(q^{-n-1} + abq^{n+2}) r^n_{11}(q^x),\\
r^n_{21}(q^{x-1})&(q^{-x} - 1) - r^n_{21}(q^{x})q^{-x}(q^{-1} + aq)
	+ r^n_{21}(q^{x+1})(aq^{-x} - abq) = -(q^{-n} + abq^{n+1}) r^n_{21}(q^x),\\
r^n_{12}(q^{x-1})&(q^{-x} - 1) - r^n_{12}(q^{x})q^{-x}(1 + a)
	+ r^n_{12}(q^{x+1})(aq^{-x} - abq) = -(q^{-n-1} + abq^{n+2}) r^n_{12}(q^x),\\
r^n_{22}(q^{x-1})&(q^{-x} - 1) - r^n_{22}(q^{x})q^{-x}(1 + a)
	+ r^n_{22}(q^{x+1})(aq^{-x} - abq) = -(q^{-n} + abq^{n+1}) r^n_{22}(q^x),
\end{align*}
that can be solved using (\ref{eqn:recrellqJP}). Using the first column of the last equation of \eqref{eqn:coefs_Qn} we obtain recurrences for 
the polynomials $\tilde{p}^n_{11}$ of degree $n$ and $\tilde{p}^n_{21}$ of degree $n-1$, which gives the first column in
\begin{align*} 
\tilde{P}_n(z) &= \begin{pmatrix}
	\kappa^n_{11} p_n(z;aq^2,b;q) &
		\kappa^n_{12} p_{n+1}(z;a,b;q) 
		+ \kappa^n_{11} (1 - z) v p_n(z;aq^2,b;q) \\
	\kappa^n_{21} p_{n-1}(z;aq^2,b;q) &
		\kappa^n_{22} p_n(z;a,b;q)
		+ \kappa^n_{21} (1 - z) v p_{n-1}(z;aq^2,b;q)
\end{pmatrix}.
\end{align*}
Since $r^n_{12}$, respectively $r^n_{22}$, are polynomial of degree $n+1$, respectively $n$, we find the explicit expression for 
$r^n_{12}$ and $r^n_{22}$ in terms of little $q$-Jacobi polynomials from \eqref{eqn:recrellqJP}, so that 
\eqref{eqn:coefs_Qn} gives the result. 

From the expression of the leading coefficient of $\tilde{P}_n$, $N_n$, the coefficients $\kappa^n_{11}$ and $\kappa^n_{12}$ are 
determined and we obtain (\ref{initial_values}). The expression of $(N_n)_{22}$  gives the relation 
\begin{align}
\label{eqn:k22_k21}
\displaystyle
\kappa^n_{22} &= (-1)^nq^{\binom{n}{2}}\frac{(aq;q)_n}{(abq^{n+1};q)_n}
	- \kappa^n_{21} v q^{n-1} 
		\frac{(1-aq)(1-aq^2)}{(1-abq^{n+1})(1-aq^{n+1})}.
\end{align}
Now we use orthogonality to determine completely $\kappa_{21}^n$ and $\kappa^n_{22}$,
\begin{align} \label{eqn:Q_orth}
\langle \tilde{P}_m, \tilde{P}_n \rangle_W 
	&= \sum_{x = 0}^{\infty} (aq)^x \frac{(bq; q)_{x}}{(q; q)_{x}}
		\left( \tilde{P}_m(q^x) A^x \right)
		\left( \tilde{P}_n(q^x) A^x \right)^* \\
	&= \sum_{x = 0}^{\infty} (aq)^x \frac{(bq; q)_{x}}{(q; q)_{x}}
		Q_m(q^x) Q_n^*(q^x)
	= H_n \delta_{m,n},  \nonumber
\end{align}
where $H_n$ is a strictly positive matrix and
\begin{align} \label{eqn:explicit_Q_prod}
Q_m(q^x) Q_n^*(q^x) &= \begin{pmatrix}
	r^m_{11}(q^x) \overline{r^n_{11}(q^x)} 
		+ r^m_{12}(q^x) \overline{r^n_{12}(q^x)} &
	r^m_{11}(q^x) \overline{r^n_{21}(q^x)} 
		+ r^m_{12}(q^x) \overline{r^n_{22}(q^x)} \\[0.3cm]
	r^m_{21}(q^x) \overline{r^n_{11}(q^x)} 
		+ r^m_{22}(q^x) \overline{r^n_{12}(q^x)} &
	r^m_{21}(q^x) \overline{r^n_{21}(q^x)} 
		+ r^m_{22}(q^x) \overline{r^n_{22}(q^x)}
\end{pmatrix}.
\end{align}
Combining (\ref{eqn:Q_orth}) with entry $(2, 1)$ of (\ref{eqn:explicit_Q_prod}) we have
\begin{align} \label{eqn:entry21}
(H_n)_{21} \delta_{m,n} &= \sum_{x = 0}^{\infty} (aq)^x \frac{(bq; q)_{x}}{(q; q)_{x}} 
	( r^m_{21} \overline{r^n_{11}} + r^m_{22} \overline{r^n_{12}} ) 
	= \langle r^m_{21}, r^n_{11} \rangle_{(a,b)}
	+ \langle r^m_{22}, r^n_{12} \rangle_{(a,b)} \\
	&= \langle 
		\kappa^m_{21} z p_{m-1}(z;aq^2,b;q),
		\kappa^n_{11} z p_{n}(z;aq^2,b;q)
	\rangle_{(a,b)}
	+
	\langle
		\kappa^m_{22} p_m(z;a,b;q),
		\kappa^n_{12} p_{n+1}(z;a,b;q)
	\rangle_{(a,b)} \nonumber \\
&= \kappa^m_{21} \overline{\kappa^n_{11}}
	\langle p_{m-1}(z;aq^2;b;q), p_n(z;aq^2,b;q) \rangle_{(aq^2,b)}
	+
	\kappa^m_{22} \overline{\kappa^n_{12}}
	\langle p_m(z;a,b;q), p_{n+1}(z;a,b;q) \rangle_{(a,b)}. \nonumber
\end{align}
Taking $(m, n) \mapsto (n, n-1)$ and using the orthogonality relations (\ref{eqn:lqjp_orth}) gives a linear
relation, which together with \eqref{eqn:k22_k21} determine $\kappa^n_{22}$ and $\kappa^n_{21}$ as 
given by (\ref{initial_values}).
This completes the proof of the theorem.
\end{proof}

\begin{corollary} \label{cor:Hn}
For the matrix-valued polynomials $(\tilde{P}_n)_{n \geq 0}$ as in the proof of Theorem \ref{thm:exp_exp} with diagonal eigenvalues we have
\begin{align*} 
\langle \tilde{P}_m, \tilde{P}_n \rangle_W = H_n \delta_{m,n},
\end{align*}
where $H_n$ is the diagonal matrix
\begin{align*}
H_n &= \diag(
	|\kappa^n_{11}|^2 h_n(aq^2, b; q) + |\kappa^n_{12}|^2 h_n(a, b; q),
	|\kappa^n_{21}|^2 h_n(aq^2, b; q) + |\kappa^n_{22}|^2 h_n(a, b; q)),
\end{align*}
and $h_n(a, b; q)$ is defined in (\ref{eqn:lqjp_orth}).
\end{corollary}

\begin{proof}
For $m = n$ (\ref{eqn:entry21}) shows $(H_n)_{21} = 0$.
Similarly we compute $(H_n)_{12} = 0$.
The entries $(1,1)$ and $(2,2)$ can be found by straightforward calculations similar to entries $(1, 2)$ and $(2, 1)$.
\end{proof}

\subsection{The matrix-valued $q$-hypergeometric equation}

Write $\tilde{P}_{i,n}$ for the $i$-th row of the matrix-valued polynomial $\tilde{P}_n$.
The equation $D\tilde{P}_n = \tilde{\Lambda}_n\tilde{P}_n$ can be written as two decoupled row equations
\begin{align} \label{eqn:rowP}
D\tilde{P}_{i,n}(z) &= \tilde{P}_{i,n}(q^{-1}z) F_{-1}(z)
	+ \tilde{P}_{i,n}(z) F_{0}(z)
	+ \tilde{P}_{i,n}(qz) F_{1}(z)
	=
	\tilde{\lambda}_{i,n} \tilde{P}_{i,n},
\end{align}
where $i = 1, 2$, $\tilde{\lambda}_{1,n} = -q^{-n-1} - abq^{n+2}$, $\tilde{\lambda}_{2,n} = -q^{-n} - abq^{n+1}$ and $\tilde{P}_{i,n}$ are the rows of the matrix polynomials $\tilde{P}_n$.
We rewrite (\ref{eqn:rowP}) by multiplying on the right by $zA$
\begin{align} \label{eqn:scalar_q_diff}
\tilde{P}_{i,n}(q^{-1}z) \left( 1 - z \right)
	+ \tilde{P}_{i,n}(z) \left(
		z\left(K-\lambda_{i,n}I\right)A - \left(I + aA^{2}\right)
	\right)
	+ \tilde{P}_{i,n}(qz)\left((a - abqz) A^{2} \right)
	=
	0.
\end{align}

\begin{proposition} \label{prop:mvlqJp-tirao}
The solution of (\ref{eqn:scalar_q_diff}) is
\begin{align} \label{eqn:mvlqJp-tirao}
\tilde{P}_{i,n}(z) &= \tilde{P}_{i,n}(0)\ 
	\peq{2}{1}{KA - \tilde{\lambda}_{i,n} A, -abqA^{2}}{aA^{2}}{q}{qz}.
\end{align}
\end{proposition}

\begin{proof}
Since $0<a<q^{-1}$ we have $\sigma(aA^{2}) \cap q^{-\NN \backslash \{ 0 \}} = \{a, aq^{2} \} \cap q^{- \NN \backslash \{ 0 \}} = \emptyset$, 
so that we can apply Theorem \ref{thm:MatqHyp} on (\ref{eqn:scalar_q_diff}) to get (\ref{eqn:mvlqJp-tirao}).
\end{proof}

Because $\tilde{P}_{i,n}$ are not only analytic row-vector-valued, but actually polynomials, we find conditions on 
$\tilde{P}_{i,n}(0)$ in order for the series (\ref{eqn:mvlqJp-tirao}) to terminate.
Writing $\tilde{P}_{i,n}(z) = \sum_{k = 0}^{\infty} G_i^k z^k$ we have
\begin{align}\label{eq:recurrenceGk1}
G_i^k = \frac{q}{1 - q^k} G_i^{k-1} (I - q^{k - 1}(K - \tilde{\lambda}_{i, n}) A + abq^{2k - 1} A^2)(I - aq^k A^2)^{-1},
\end{align}
and we must have $G_i^n \neq 0$ and $G_i^{n+1} = 0$.
Therefore
\begin{align}\label{eq:recurrenceGk}
(G_i^n)^t \in \ker\bigl( (I - q^n(K - \tilde{\lambda}_{i,n} A) + abq^{2n + 1}A^2)^t\bigr).
\end{align}
The matrix is upper triangular and the $(1,1)$-entry vanishes for $\lambda_{1,n}$ and the $(2,2)$-entry vanishes for $\lambda_{2,n}$.  
Using the definition of $G_i^k$ we can determine $G_2^0$ completely up to a scalar, since all the matrices in 
\eqref{eq:recurrenceGk1} are invertible for $1\leq k \leq n$. 
Because $\tilde{\lambda}_{1, n-1} = \tilde{\lambda}_{2, n}$ it is not possible to determine $G_1^0$, since the kernel in \eqref{eq:recurrenceGk} 
is also non-trivial for $n$ replaced by $n-1$. .
However adding the orthogonality relation (\ref{eqn:Q_orth}), we can determine $G_1^0$.
Therefore the coefficients of $\tilde{P}_{i,n}(0)$ are completely determined up to a scalar by the fact that 
it is a orthogonal polynomial solution to \eqref{eqn:scalar_q_diff}. 

Proposition \ref{prop:mvlqJp-tirao} gives a way to write the orthogonal polynomials in a closed form.
The matrix-valued basic hypergeometric series expression of the polynomials might be useful to generalize the polynomials to higher dimensions.

\subsection{The three term recurrence relation and the Rodrigues formula}
The first goal in this section is to find the three term recurrence relation for $\tilde{P}_n$;
\begin{align} \label{eqn:3term}
z \tilde{P}_n(z) = A_n \tilde{P}_{n+1}(z) + B_n \tilde{P}_n(z)
	+ C_n \tilde{P}_{n-1}(z).
\end{align}
By comparing the leading coefficients of (\ref{eqn:3term}) we read off
\begin{align*}
A_n &= N_n N_{n+1}^{-1} = \begin{pmatrix}
1 & -\cfrac{q^n(1 - q)(1 - aq)(1 + abq^{2n+3})}{(abq^{2n+2}; q^2)_2} v \\[0.3cm]
0 & 1
\end{pmatrix}.
\end{align*}
By a well-known argument 
\begin{align*}
C_n &= \langle \tilde{P}_n, \tilde{P}_n \rangle_W A_{n-1}^*
	\langle \tilde{P}_{n-1}, \tilde{P}_{n-1} \rangle_W^{-1}.
\end{align*}
Therefore by Corollary \ref{cor:Hn} we can write $C_n = H_n A_{n-1}^* H_{n-1}^{-1}$.
To find $B_n$ we first remark that 
\begin{align*}
\tilde{P}_n(0) = \begin{pmatrix}
\kappa^n_{11} & \kappa^n_{12} + \kappa^n_{11} v \\
\kappa^n_{21} & \kappa^n_{22} + \kappa^n_{21} v
\end{pmatrix}
\end{align*}
and $\det(\tilde{P}_n(0)) = \kappa^n_{11} \kappa^n_{22} - \kappa^n_{21} \kappa^n_{12} > 0$, because both terms are positive by Theorem \ref{thm:exp_exp}. 
If we plug in $z = 0$ in (\ref{eqn:3term}) we find
\begin{align*}
B_n = -A_n \tilde{P}_{n+1}(0) (\tilde{P}_n(0))^{-1} 
	- C_n \tilde{P}_{n-1}(0) (\tilde{P}_n(0))^{-1}.
\end{align*}

Theorem \ref{thm:Rodrigues}  gives a Rodrigues formula for the matrix-valued little $q$-Jacobi polynomials.
\begin{theorem} \label{thm:Rodrigues} 
The expression
\begin{align} \label{eqn:PnRodrigues}
P_n(x) = q^{-x} D_{q}^{n} \left( \cfrac{
	a^x q^{(n+1)x}(bq;q)_{x}
}{
	(q;q)_{x-n}} T(q^{x}) R(n) T(q^{x})^*
\right) W(q^x)^{-1},
\end{align}
defines a sequence of matrix-valued orthogonal polynomials with respect to 
\eqref{eqn:q-weight} with weight matrix (\ref{eqn:mvlqjp-measure}), where
\begin{align*}
R(n) = \begin{pmatrix}
\cfrac{
	(1 - aq^{n+2})(1 - abq^{n+3}) + av^{2}q^{2}(1 - q^{n})(1 - bq^{n+1}) 
}{
	1 - abq^{2 n+3}
} & 0 \\[2mm]
-(1 - q^{n}) avq^{2} & 1 - aq^{n+2} 
\end{pmatrix}.
\end{align*}
\end{theorem}

\begin{proof}
Since the proof contains a couple of lengthy but direct calculations, we only give a sketch and leave the details to the reader.\\

To see that (\ref{eqn:PnRodrigues}) defines a family of orthogonal polynomials two things need to be proved. 
(1) For all $n \geq 0$,  (\ref{eqn:PnRodrigues}) defines a matrix-valued polynomial of degree $n$ with non-singular coefficients. 
(2) The polynomials defined by (\ref{eqn:PnRodrigues}) are orthogonal with respect to the $q$-weight given by (\ref{eqn:mvlqjp-measure}).

\textit{First step.}\\
Let us write $q^x W(q^x) = \rho(q^x) T(q^x) T(q^x)^*$, where $\rho(q^x)$ is the weight associated to the scalar little $q$-Jacobi polynomials with parameters $a$ and $b$. 
Using the $q$-Leibniz rule (\ref{eqn:q-Leibniz}) and that $T(q^{x})R(n)T(q^{x})^*$ is a matrix-valued polynomial of degree $2$ in $q^x$ we can write 
\begin{align}\label{eqn:n:mat}
D_{q}^{n} & \left( \cfrac{
	a^x q^{(n+1)x}(bq;q)_{x}
}{
	(q;q)_{x-n}
} T(q^{x}) R(n) T(q^{x})^* \right) W(q^x)^{-1}  \\ \nonumber
&= D_q^{n} \left( \cfrac{
	a^x q^{(n+1)x}(bq;q)_{x}
}{
	(q;q)_{x-n}
} \right) (\rho(x))^{-1}
T(q^x) R(n) T(q^x)^{*}(T(q^x) T(q^x)^*)^{-1} \\ \nonumber
&\quad + \qbin{n}{1}_q D_q^{n-1} \left( \cfrac{
	a^{x+1} q^{(n+1)(x+1)}(bq;q)_{x+1}
}{
	(q;q)_{x-n+1}
} \right) (\rho(x))^{-1} D_{q} \left(
	T(q^x) R(n) T(q^x)^{*}
\right) (T(q^x)T(q^x)^*)^{-1} \\ \nonumber
&\quad + \qbin{n}{2}_q D_q^{n-2} \left( \cfrac{
	a^{x+2} q^{(n+1)(x+2)}(bq;q)_{x+2}
}{
	(q;q)_{x-n+2}
} \right)(\rho(x))^{-1} D_{q}^{2} \left(
	T(q^x) R(n) T(q^x)^{*}
\right) (T(q^x) T(q^x)^*)^{-1}.
\end{align}

With the use of (\ref{eqn:Dqn:explicit}) and some lengthy calculations we can find polynomials $t_n, r_n$ and $s_n$ of degree $n$ in $q^x$ such that
\begin{align*}
t_n(q^x) &=
D_q^{n} \left( \cfrac{
	a^x q^{(n+1)x} (bq;q)_{x}
}{
	(q;q)_{x-n}
} \right) \rho(q^x)^{-1}, \\
q^x r_n(q^x) &=
D_q^{n-1} \left( \cfrac{
	a^{x+1} q^{(n+1)(x+1)}(bq;q)_{x+1}
}{
	(q;q)_{x+1-n}
} \right) \rho(q^x)^{-1}, \\
q^{2x} s_n(q^x) &=
D_q^{n-2} \left( \cfrac{
	a^{x+2} q^{(n+1)(x+2)} (bq;q)_{x+2}
}{
	(q;q)_{x+2-n}
} \right) \rho(q^x)^{-1}.
\end{align*}
Let us now focus on the matrix part of (\ref{eqn:n:mat}). By using the $q$-Leibniz rule, (\ref{eqn:Dqn:explicit}) and the fact that $T(q^{x+1}) = T(q^x) A$ we can write  
\begin{align*}
D_q(T(q^x) R(n) T(q^x)^*) (T(q^x)^*)^{-1} T(q^x)^{-1}
&= 
\cfrac{1}{(1-q)q^x} T(q^x) R_1(n) T(q^x)^{-1} \\
D^2_q(T(q^x) R(n) T(q^x)^*) (T(q^x)^*)^{-1} T(q^x)^{-1}
&=
\cfrac{1}{(1-q)^2 q^{2x}}T(q^x) R_2(n) T(q^x)^{-1},
\end{align*}
where 
\begin{align*}
R_1(n) &= R(n) - A R(n) A^*, \quad
R_2(n) = R(n) - (1+q^{-1}) A R(n) A^* + q^{-1} A^2 R(n) (A^*)^2.
\end{align*}
In (\ref{eqn:n:mat}) we can now write 
\begin{align*}
\left(
	T(q^x) R(n) T(q^x)^* 
\right)
(T(q^x)T(q^x)^*)^{-1}
&= 
q^{-x} A_0(n) + B_0(n) + q^x C_0(n) \\
D_{q} \left( 
	T(q^x) R(n) T(q^x)^{*} 
\right)
(T(q^x)T(q^x)^*)^{-1}
&=
q^{-2x} A_1(n) + q^{-x} B_1(n) + C_1 (n), \\
D_{q}^{2} \left( 
	T(q^x) R(n) T(q^x)^{*} 
\right)
(T(q^x) T(q^x)^*)^{-1} 
&=
q^{-3x} A_2(n) + q^{-2x} B_2(n) + q^{-x} C_2(n).
\end{align*}
Tedious, although straightforward calculations, show that $t_{n}^{0} A_0(n) + r_{n}^{0} A_1(n) + s_n^0 A_2(n) = 0$, 
$t_{n}^{n} C_0(n) + r_{n}^{n} C_1(n) + s_n^n C_2(n) = 0$ 
and $t_{n}^{n} B_0(n) + r_{n}^{n} B_1(n) + s_n^n B_2(n) + r_n^{n-1} C_1(n) + s_n^{n-1} C_2(n)$ is non-singular.
This shows that (\ref{eqn:PnRodrigues}) is a matrix-valued polynomial of degree $n$ with non-singular leading coefficient.

\textit{Second step.}\\
To prove that the sequence of polynomials given by (\ref{eqn:PnRodrigues}) is orthogonal, we must prove that for $n \geq 1$ and $0 \leq m < n$, $\langle P_n, x^m I \rangle_{W} = 0$ holds.\\
In order to prove this we use Lemma \ref{claim_bd}, which will be proved later.
\begin{lemma}\label{claim_bd}
For $1 < k < n$, 
\begin{align*}
D_{q}^{n-k} \left( \cfrac{
	a^{x+k-1} q^{(n+1)(x+k-1)} (bq;q)_{x+k-1}
}{
	(q;q)_{x+k-n-1}
} T(q^{x+k-1}) R(n) T(q^{x+k-1})^* \right) 
D_{q}^{k}(q^{xm})
\end{align*}
is zero for $x = 0$ and $x \to \infty$.
\end{lemma}

By using the $q$-Leibniz rule (\ref{eqn:q-Leibniz}), the formal identity given in (\ref{eqn:q-fund_calc}) and Lemma \ref{claim_bd}, we get
\begin{align*}
\langle P_{n}, z^m \rangle_{W} 
&=
\sum_{x=0}^{\infty} D_{q}^{n} \left( \cfrac{
	a^x q^{(n+1)x} (bq;q)_{x}
}{
	(q;q)_{x-n}
} T(q^{x}) R(n) T(q^{x})^* \right) q^{xm} \\
&=
D_{q}^{n-1} \left( \cfrac{
	a^{x} q^{(n+1)x} (bq;q)_{x}
}{
	(q;q)_{x-n}
} T(q^{x}) R(n) T(q^{x})^* \right) D_q(q^{xm})|_{x = 0}^{\infty} \nonumber\\
&\quad +
\sum_{x=0}^{\infty} D_{q}^{n-1} \left( \cfrac{
	a^{x+1} q^{(n+1)(x+1)} (bq;q)_{x+1}
}{
	(q;q)_{x+1-n}
} T(q^{x}) R(n) T(q^{x})^* \right) D_{q}(q^{xm})\nonumber\\
&= \sum_{x=0}^{\infty} D_{q}^{n-1} \left( \cfrac{
	a^{x+1} q^{(n+1)(x+1)} (bq;q)_{x+1}
}{
	(q;q)_{x+1-n}
} T(q^{x}) R(n) T(q^{x})^* \right) D_{q}(q^{xm}),
\end{align*}
By repeating this process we obtain 
\begin{align*}
\langle P_{n}, z^m \rangle
&=
\sum_{x=0}^{\infty} D_{q}^{n-m-1} \left( \cfrac{
	a^{x+m+1} q^{n(x+m+2)} (bq;q)_{x+m+1}
}{
	(q;q)_{x+m+1-n}
} T(q^{x}) R(n) T(q^{x})^* \right) D_{q}^{m+1}(q^{xm}) q^{x} 
= 0
\end{align*}
because $D_{q}^{m+1}(q^{xm}) = 0$.
This gives the desired result.
\end{proof}

\begin{proof}[Proof of Lemma \ref{claim_bd}.]
To see that the first boundary condition at $x = 0$ holds, we use the expression 
\begin{align*}
\frac{1}{(q;q)_{x-n}} = \frac{(q^{x-n+1};q)_{n}}{(q;q)_{x}},
\end{align*}
which vanishes at $x=0$. Any other quantity involved is bounded in $x=0$, hence Lemma \ref{claim_bd} holds in this case.

For $x \to \infty$, use that $a < q^{-1}$ and so $a^{x+k-1} q^{(n+1)(x+k-1)}$ tends to $0$ when $x$ tends to $\infty$. 
Since all the other quantities remain bounded when $x$ tends to $\infty$, we obtain the desired result.  
\end{proof}

\section{Acknowledgement}
At various opportunities parts of this paper have been discussed with Antonio Dur\'an, Kerstin Jordaan, Pablo Rom\'an, and we thank them for their input. 
The research of Noud Aldenhoven is supported by the Netherlands Organisation for Scientific Research (NWO) under project number \textbf{613.001.005} and by Belgian Interuniversity Attraction Pole Dygest \textbf{P07/18} .\\
The research of Ana Mart\'inez de los R\'ios is partially supported by  \textbf{MTM2009-12740-C03-02}, \par \textbf{MTM2012-36732-C03-03} (Ministerio de Econom\'ia y Competitividad. Gobierno de Espa\~{n}a),\textbf{ FQM-4643} (Junta de Andaluc\'ia). \\
We thank the referees for their valuable comments which improved the paper significantly.

\bibliographystyle{plain}

\end{document}